\documentclass[11pt]{amsart}

\usepackage[margin=1.1in]{geometry}
\usepackage{amsmath,amssymb,amsfonts,amsthm,mathtools}
\usepackage{mathrsfs}
\usepackage{enumitem}
\usepackage[colorlinks=true,citecolor=blue,linkcolor=blue,urlcolor=blue]{hyperref}

\newtheorem{theorem}{Theorem}[section]
\newtheorem{lemma}[theorem]{Lemma}
\newtheorem{proposition}[theorem]{Proposition}
\newtheorem{corollary}[theorem]{Corollary}
\newtheorem{conjecture}[theorem]{Conjecture}
\theoremstyle{remark}

\DeclareMathOperator{\tr}{tr}
\DeclareMathOperator{\spanop}{span}
\newcommand{\HH}{\mathbb H}
\newcommand{\RR}{\mathbb R}
\newcommand{\Sn}{\mathbb S}
\newcommand{\Id}{I}
\newcommand{\mcM}{\mathsf M}
\newcommand{\mcK}{\mathsf K}

\title[Reciprocal sums]{Reciprocal sums of Neumann eigenvalues in non-Euclidean space forms}
\author{Jiangcheng You and Heng Zhang}
\address{School of Mathematical Sciences, University of Science and Technology of China, Hefei, China}
\email{yjcmp@mail.ustc.edu.cn}

\address{School of Mathematical Sciences, University of Science and Technology of China, Hefei, China}
\email{hengz@mail.ustc.edu.cn}
\date{\today}

\begin{document}

\begin{abstract}
Let $M^n_\kappa$ be the simply connected space form of dimension $n\ge2$ and constant sectional curvature $\kappa\in\{-1,1\}$.  For every bounded connected smooth domain $\Omega\subset M^n_\kappa$, assume in the case $\kappa=1$ that $\Omega$ is contained in an open hemisphere, and let $B_\Omega$ be a geodesic ball with $|B_\Omega|=|\Omega|$.  We prove
$$
        \sum_{j=1}^n \frac1{\mu_j(\Omega)}\ge \frac{n}{\mu_1(B_\Omega)},
$$
where $\mu_j(\Omega)$ are the positive Neumann eigenvalues of $\Omega$.  Equality holds if and only if $\Omega$ is a geodesic ball.  This proves a conjecture proposed by Xia and Wang [Math. Ann. 385, 2023, 863-879].
\end{abstract}

\maketitle

\section{Introduction}
\subsection{Background and main result}
Let $\Omega$ be a bounded connected domain with smooth boundary in a Riemannian manifold.  We use the sign convention
$$
        \Delta u=-\mu u,
        \qquad
        \partial_\nu u=0 \quad \text{on }\partial\Omega,
$$
and write the Neumann spectrum as
$$
        0=\mu_0(\Omega)<\mu_1(\Omega)\le \mu_2(\Omega)\le\cdots .
$$

The isoperimetric theory of Neumann eigenvalues begins with the Szeg\H{o}--Weinberger inequality.  Szeg\H{o} proved in dimension two that, among simply connected planar domains of prescribed area, the disk maximizes the first nonzero Neumann eigenvalue \cite{Szego}.  Weinberger removed the planar and conformal restrictions and proved the $n$-dimensional Euclidean inequality
$$
        \mu_1(\Omega)
        \le \left(\frac{\omega_n}{|\Omega|}\right)^{2/n}p_{n/2,1}^2,
        \qquad \Omega\subset\RR^n,
$$
with equality if and only if $\Omega$ is a ball \cite{Weinberger}.  Here $\omega_n$ denotes the volume of the Euclidean unit ball and $p_{n/2,1}$ is the first positive zero of the derivative of $x^{1-n/2}J_{n/2}(x)$.

Ashbaugh and Benguria then studied reciprocal sums of low Neumann eigenvalues.  They proved the scale-sharp but non-optimal estimate
$$
        \sum_{j=1}^n\frac1{\mu_j(\Omega)}
        \ge \frac{n}{n+2}\left(\frac{|\Omega|}{\omega_n}\right)^{2/n},
        \qquad \Omega\subset\RR^n,
$$
and formulated the following sharp conjecture \cite{AB1993}.

\begin{conjecture}[Ashbaugh--Benguria]\label{conj:AB-euclidean}
For any bounded domain $\Omega\subset\RR^n$ with smooth boundary,
\begin{equation*}
        \sum_{j=1}^n\frac1{\mu_j(\Omega)}
        \ge
        \frac{n(|\Omega|/\omega_n)^{2/n}}{p_{n/2,1}^2}.
\end{equation*}
Equality holds if and only if $\Omega$ is a ball in $\RR^n$.
\end{conjecture}

Conjecture \ref{conj:AB-euclidean} was subsequently highlighted by Ashbaugh and by Henrot \cite{AshbaughOpen,Henrot}.  Related reciprocal-eigenvalue estimates and quantitative refinements include work of Hile--Xu, Nadirashvili, and Brasco--Pratelli \cite{HileXu,Nadirashvili,BrascoPratelli}.  A sharp partial result toward Conjecture \ref{conj:AB-euclidean} was obtained by Xia and Wang, who proved the $(n-1)$-term inequality
$$
        \sum_{j=1}^{n-1}\frac1{\mu_j(\Omega)}
        \ge
        \frac{(n-1)(|\Omega|/\omega_n)^{2/n}}{p_{n/2,1}^2},
        \qquad \Omega\subset\RR^n,
$$
with equality only for balls \cite{XiaWang}.  Recently, He, Li and Tang posted a proof of Conjecture \ref{conj:AB-euclidean}; a central feature of their argument is to retain the transplanted first eigenspace as one coupled finite-dimensional trial space rather than to estimate its coordinate functions separately \cite{HeLiTang}.

The Szeg\H{o}--Weinberger inequality for $\mu_1$ has also been studied in spaces of constant curvature.  In the hyperbolic case, it was proved independently by Ashbaugh--Benguria and by Xu; Chavel observed that Weinberger's method can be adapted to this setting \cite{AB1995,Xu,Chavel}.  Ashbaugh--Benguria also proved the corresponding sharp upper bound for domains contained in a hemisphere of the sphere \cite{AB1995}.  Motivated by these extensions, Xia and Wang formulated the following  non-Euclidean space-form analogue of Conjecture \ref{conj:AB-euclidean} for constant curvature $\kappa\in\{-1,1\}$ \cite{XiaWang}.

\begin{conjecture}[Xia--Wang]\label{conj:XW-spaceform}
Let $M^n_\kappa$ be the simply connected $n$-dimensional Riemannian manifold of constant sectional curvature $\kappa\in\{-1,1\}$, and let $\Omega\subset M^n_\kappa$ be a bounded connected domain with smooth boundary.  In the case $\kappa=1$, assume that $\Omega$ is contained in an open hemisphere.  Let $B_\Omega\subset M^n_\kappa$ be a geodesic ball such that $|B_\Omega|=|\Omega|$.  Then
\begin{equation}\label{eq:XW-spaceform-conj}
        \sum_{j=1}^n\frac1{\mu_j(\Omega)}
        \ge
        \frac{n}{\mu_1(B_\Omega)}.
\end{equation}
Equality holds if and only if $\Omega$ is isometric to $B_\Omega$.
\end{conjecture}

Xia and Wang proved the corresponding sharp $(n-1)$-term inequalities in the Euclidean and hyperbolic cases; in particular, for $\Omega\subset\HH^n$ they obtained
$$
        \sum_{j=1}^{n-1}\frac1{\mu_j(\Omega)}
        \ge
        \frac{n-1}{\mu_1(B_\Omega)},
$$
with equality only for geodesic balls \cite{XiaWang}.  In positive curvature, Benguria, Brandolini and Chiacchio proved the corresponding harmonic-mean inequality for the first $n-1$ nontrivial Neumann eigenvalues on domains contained in a hemisphere of $\Sn^n$ \cite{BBC}.  Related lower-order estimates have also been developed in rank-one symmetric spaces \cite{MengWang}.

The main result of this paper proves the conjecture of Xia--Wang.

\begin{theorem}\label{thm:main}
Conjecture \ref{conj:XW-spaceform} holds true.  
\end{theorem}

% \begin{remark}\label{rem:connectedness}
% The connectedness assumption is essential.  If disconnected bounded open sets are allowed and the symbols $\mu_j$ denote the positive Neumann eigenvalues after removing the whole zero eigenspace, the inequality fails.  Indeed, fix a hyperbolic ball $B_R$ and let $\Omega_m$ be the disjoint union of $m$ geodesic balls of equal volume $|B_R|/m$.  If $r_m$ is their radius, then the first $n$ positive Neumann eigenvalues of $\Omega_m$ are all equal to $\mu_1(B_{r_m})$, while the comparison ball for $\Omega_m$ is $B_R$.  Since $r_m\to0$ and $\mu_1(B_{r_m})\to\infty$, for $m$ large
% $$
%         \sum_{j=1}^n \frac1{\mu_j(\Omega_m)}
%         =
%         \frac{n}{\mu_1(B_{r_m})}
%         <
%         \frac{n}{\mu_1(B_R)}.
% $$
% Thus no disconnected version of the theorem is possible under this convention.
% \end{remark}

\subsection{Proof strategy}
The proof is a finite-dimensional, matrix-valued refinement of
Weinberger's transplantation method.  In the classical coordinatewise
Weinberger argument, and in the known sharp $(n-1)$-term estimates, the
transplanted coordinate functions are tested essentially one at a time.  The
key point here is instead to keep the whole transplanted first eigenspace as
a single coupled $n$-dimensional trial space, and to estimate its mass and
stiffness forms simultaneously in the Loewner order.  This is structurally
analogous to the recent Euclidean proof of He--Li--Tang \cite{HeLiTang}.  The
present paper develops the corresponding argument in space forms; the main
new difficulties are the curvature-dependent radial monotonicity estimates
and, in the spherical case, the need to fold antipodally before applying
raywise rearrangement.

More precisely, we first choose a Weinberger center and transplant the
$n$-dimensional first nonzero Neumann eigenspace of the comparison ball
$B_\Omega$.  This gives trial functions $P_1,\dots,P_n$ with zero mean on
$\Omega$.  Let $\mcM$ and $\mcK$ be their mass and stiffness matrices, as
defined in \eqref{eq:hyperbolic-MK-def}.  The trace form of the Ritz principle
(Lemma \ref{lem:trace-ritz}) gives
$$
        \sum_{j=1}^n\frac1{\mu_j(\Omega)}
        \ge \tr(\mcK^{-1}\mcM).
$$
Thus the problem is reduced to obtaining a sharp lower bound for
$\tr(\mcK^{-1}\mcM)$.

The required bound comes from raywise rearrangement in geodesic polar
coordinates about the Weinberger center.  On each ray, the radial
monotonicity properties of the transplanted profile compare the contribution
of $\Omega$ with that of the centered ball of the same volume.  After
integrating over the angular variable, these one-dimensional comparisons
become the matrix inequalities
$$
        \mcM\succeq a\Id+cZ,
        \qquad
        \mcK\preceq \lambda a\Id-\eta Z,
        \qquad
        \tr Z=0,
$$
where $\lambda=\mu_1(B_\Omega)$, the constants $a,c,\eta$ are positive, and
$Z$ is a symmetric trace-free matrix measuring the angular imbalance of
$\Omega$ relative to the chosen center.  The important feature is that the
same imbalance matrix appears in the two estimates with opposite signs.
Lemma \ref{lem:matrix} then converts these Loewner-order bounds directly into
$$
        \tr(\mcK^{-1}\mcM)\ge \frac{n}{\lambda},
$$
which yields the desired reciprocal-sum inequality.

In the spherical case an additional geometric issue has to be addressed.
Although $\Omega$ is assumed to lie in some open hemisphere, the Weinberger
center produced by the centering argument need not be the center of that
hemisphere.  To remove this obstruction, we use a radial profile, with
respect to the chosen center $p$, that is even under $t\mapsto \pi-t$.  The
corresponding trial functions are odd under the antipodal map, while their
mass and stiffness densities are antipodally invariant.  Hence the part of
$\Omega$ lying beyond the equator of $p$ may be folded back into
$B_{\pi/2}(p)$ without changing the matrices $\mcM$ and $\mcK$.  After this
folding, the entire rearrangement argument takes place in the half-sphere
where $\sin t$ is increasing, so that the same matrix mechanism as in the
hyperbolic case applies.

\subsection{Organization of the paper}
Necessary preliminary tools are assembled in Section \ref{sec2}. Section \ref{sec3} deals with the hyperbolic case, and Section \ref{sec4} treats the spherical case.

\section{Common preliminaries}\label{sec2}

We write
$$
        B_R(p)=\{x:d(p,x)<R\}
$$
for the geodesic ball of radius $R$ centered at $p$.  When the center is irrelevant, $B_R$ denotes an arbitrary geodesic ball of radius $R$.  For $\xi\in\RR^n$, the notation $\xi\otimes\xi$ denotes the rank-one matrix $(\xi_i\xi_j)_{i,j=1}^n$.

\subsection{Trace Ritz principle}

Let
$$
        H^1_{\perp}(\Omega)=\left\{u\in H^1(\Omega):\int_\Omega u\,dv=0\right\}.
$$
The positive Neumann spectrum is the spectrum of the Dirichlet form $\int_\Omega |\nabla u|^2dv$ on $H^1_{\perp}(\Omega)$ with mass form $\int_\Omega u^2dv$.

Let $\Omega$ be a bounded connected smooth domain, and let $P_1,\dots,P_n\in H^1_{\perp}(\Omega)$ be linearly independent.  Define the mass matrix $\mcM=(\mcM_{ij})$ and stiffness matrix $\mcK=(\mcK_{ij})$ as
\begin{equation}\label{eq:hyperbolic-MK-def}
        \mcM_{ij}=\int_\Omega P_iP_j\,dv,
        \qquad
        \mcK_{ij}=\int_\Omega \langle\nabla P_i,\nabla P_j\rangle\,dv.
\end{equation}

\begin{lemma}\label{lem:trace-ritz}
$\mcM$ and $\mcK$ are positive definite and
\begin{equation}\label{eq:trace-ritz}
        \sum_{j=1}^n\frac1{\mu_j(\Omega)}
        \ge \tr(\mcK^{-1}\mcM).
\end{equation}
\end{lemma}

\begin{proof}
Let
$$
        V=\spanop\{P_1,\dots,P_n\}\subset H^1_{\perp}(\Omega).
$$
Positive definiteness of $\mcM$ follows from linear independence.  If a nonzero $u\in V$ had zero stiffness, then $\nabla u=0$ a.e.  Since $\Omega$ is connected, $u$ would be constant.  Because $u\in H^1_{\perp}(\Omega)$, this constant would be zero, contradicting linear independence.  Hence $\mcK$ is positive definite.

The Ritz values $\rho_1\le\cdots\le\rho_n$ of the pair $(\mcK,\mcM)$ are characterized by
$
        \mcK v=\rho\,\mcM v 
$.  By the min-max principle,
$$
        \mu_j(\Omega)\le \rho_j,
        \qquad j=1,\dots,n.
$$
Consequently
$$
        \sum_{j=1}^n\frac1{\mu_j(\Omega)}
        \ge \sum_{j=1}^n\frac1{\rho_j}.
$$
The reciprocal Ritz values are the eigenvalues of $\mcK^{-1/2}\mcM\mcK^{-1/2}$; hence their sum is $\tr(\mcK^{-1}\mcM)$.
\end{proof}

\subsection{One-dimensional rearrangement on rays}

Let $L\in(0,\infty]$, let $w:(0,L)\to(0,\infty)$ be locally integrable, and put
$$
        \mathcal V(t)=\int_0^t w(\tau)\,d\tau.
$$
When $L=\infty$, $\mathcal V(L)$ means $\lim_{t\to\infty}\mathcal V(t)$.  For $y\in[0,\mathcal V(L))$, let $\tau(y)$ be the inverse function of $\mathcal V$. We shall use the following elementary one-dimensional rearrangement fact.

\begin{lemma}\label{lem:ray-rearrangement}
Let $E\subset(0,L)$ be measurable and let
$$
        y=\int_Ew(t)\,dt<\mathcal V(L).
$$
Let $\psi:(0,L)\to\RR$ be measurable and monotone.    If $\psi$ is nondecreasing, then
$$
        \int_E\psi(t)w(t)\,dt
        \ge \int_0^{\tau(y)}\psi(t)w(t)\,dt.
$$
If $\psi$ is nonincreasing, then
$$
        \int_E\psi(t)w(t)\,dt
        \le \int_0^{\tau(y)}\psi(t)w(t)\,dt.
$$
The same conclusions hold when $y=\mathcal V(L)<\infty$, with $\tau(y)=L$.
\end{lemma}

\begin{proof}
Since
$$
        \int_E w(t)\,dt
        =
        \int_0^{\tau(y)} w(t)\,dt
        =y,
$$
we have
$$
        \int_{E\setminus(0,\tau(y))}w(t)\,dt
        =
        \int_{(0,\tau(y))\setminus E}w(t)\,dt .
$$
If \(\psi\) is nondecreasing, then
$$
\begin{aligned}
&\int_E\psi(t)w(t)\,dt
-\int_0^{\tau(y)}\psi(t)w(t)\,dt  \\
&\quad =
\int_{E\setminus(0,\tau(y))}
        \bigl(\psi(t)-\psi(\tau(y))\bigr)w(t)\,dt
+
\int_{(0,\tau(y))\setminus E}
        \bigl(\psi(\tau(y))-\psi(t)\bigr)w(t)\,dt
\ge0,
\end{aligned}
$$
because \(t\ge\tau(y)\) on \(E\setminus(0,\tau(y))\) and
\(t\le\tau(y)\) on \((0,\tau(y))\setminus E\).  The nonincreasing case is obtained by reversing the
inequality in the same argument, and the endpoint cases \(y=0\) and
\(y=\mathcal V(L)<\infty\) are immediate.
\end{proof}

\subsection{A matrix lemma}

We adopt the Loewner order on the space of real symmetric matrices. For two such matrices \(X,Y\), we write \(X \succeq Y\) if \(X-Y\) is positive semidefinite, and \(X \preceq Y\) if \(Y-X\) is positive semidefinite.

\begin{lemma}\label{lem:matrix}
Let $\mcM$ and $\mcK$ be positive definite symmetric $n\times n$ matrices.  Suppose that $a,c,\eta,\lambda>0$ and that $Z$ is a symmetric trace-free matrix such that
\begin{equation}\label{eq:matrix-hyp}
        \mcM\succeq a\Id+cZ,
        \qquad
        \mcK\preceq \lambda a\Id-\eta Z.
\end{equation}
Then
\begin{equation}\label{eq:matrix-conclusion}
        \tr(\mcK^{-1}\mcM)\ge \frac{n}{\lambda}.
\end{equation}
Moreover, if equality holds in \eqref{eq:matrix-conclusion}, then
\begin{equation}\label{eq:matrix-equality-conclusion}
        Z=0,
        \qquad
        \mcM=a\Id,
        \qquad
        \mcK=\lambda a\Id.
\end{equation}
\end{lemma}

\begin{proof}
Set
$$
        T=\lambda a\Id-\eta Z.
$$
Since $T-\mcK\succeq0$ and $\mcK>0$, $T$ is positive definite.
We have
\begin{equation}\label{eq:matrix-chain}
        \tr(\mcK^{-1}\mcM)
        \ge \tr(T^{-1}\mcM)
        \ge \tr\bigl(T^{-1}(a\Id+cZ)\bigr).
\end{equation}
Indeed, if $X,Y\succeq0$, then $\tr(XY)=\tr(Y^{1/2}XY^{1/2})\ge0$.

Diagonalize $Z$.  Let $z_1,\dots,z_n$ be its eigenvalues.  Since $T=\lambda a\Id-\eta Z$ and $a\Id+cZ$ are affine functions of $Z$, they are diagonal in the same orthonormal basis, and so is $T^{-1}$.  
The positivity of $T$ gives $\lambda a-\eta z_i>0$ for every $i$.  Hence
\begin{align*}
        \tr\bigl(T^{-1}(a\Id+cZ)\bigr)
        &=\sum_{i=1}^n\frac{a+cz_i}{\lambda a-\eta z_i}\\
        &=\frac{n}{\lambda}
        +\frac{\lambda c+\eta}{\lambda}
        \sum_{i=1}^n\frac{z_i}{\lambda a-\eta z_i}.
\end{align*}
The function
$$
        \phi(z)=\frac{z}{\lambda a-\eta z}
$$
is strictly convex on $(-\infty,\lambda a/\eta)$, because
$$
        \phi''(z)=\frac{2\lambda a\eta}{(\lambda a-\eta z)^3}>0.
$$
Jensen's inequality yields
$$
        \sum_{i=1}^n\phi(z_i)
        \ge n\phi\left(\frac1n\sum_{i=1}^nz_i\right)=0.
$$
This proves \eqref{eq:matrix-conclusion}.

Assume now that equality holds in \eqref{eq:matrix-conclusion}.  Then equality must hold in every step of \eqref{eq:matrix-chain} and in Jensen's inequality.  Since $\phi$ is strictly convex, equality in Jensen's inequality forces $z_1=\cdots=z_n$.  Since $\sum_i z_i=0$, this gives $Z=0$.

With $Z=0$, we have $T=\lambda a\Id$.  Equality in the first inequality of \eqref{eq:matrix-chain} gives
$$
        \tr\bigl((\mcK^{-1}-T^{-1})\mcM\bigr)=0.
$$
Here $\mcK^{-1}-T^{-1}\succeq0$ and $\mcM>0$, so $\mcK^{-1}=T^{-1}$ and hence $\mcK=T=\lambda a\Id$.  Equality in the second inequality of \eqref{eq:matrix-chain} similarly gives
$$
        \tr\bigl(T^{-1}(\mcM-a\Id)\bigr)=0.
$$
Since $T^{-1}>0$ and $\mcM-a\Id\succeq0$, we get $\mcM=a\Id$.  This proves \eqref{eq:matrix-equality-conclusion}.
\end{proof}

\subsection{Almost-everywhere equality of smooth domains}

\begin{lemma}\label{lem:ae-domain-equality}
Let $D_1$ and $D_2$ be open domains with smooth boundary in a smooth Riemannian manifold.  If $D_1$ and $D_2$ agree up to a null set, then $D_1=D_2$.
\end{lemma}

\begin{proof}
Suppose that $x\in D_1\setminus D_2$.  Since $D_1$ is open, a small geodesic ball $U$ centered at $x$ is contained in $D_1$.  If $x$ is an interior point of the complement of $D_2$, then $U\setminus D_2$ has positive volume for $U$ small.  If $x\in\partial D_2$, the smoothness of $\partial D_2$ gives local coordinates in which $D_2$ is a half-ball; hence $U\setminus D_2$ again has positive volume.  This contradicts $|D_1\setminus D_2|=0$.  Thus $D_1\subset D_2$.  Interchanging the roles of $D_1$ and $D_2$ gives the reverse inclusion.
\end{proof}

\section{The hyperbolic case}\label{sec3}

Throughout this section,
$$
        s(t)=\sinh t,
        \qquad
        w(t)=s(t)^{n-1}.
$$
About any point $p\in\HH^n$, geodesic polar coordinates identify $\HH^n\setminus\{p\}$ with $(0,\infty)\times\Sn^{n-1}$ and the metric is
\begin{equation}\label{eq:hyperbolic-polar-metric}
        ds_{\HH^n}^2=dt^2+s(t)^2d\sigma^2,
\end{equation}
where $d\sigma^2$ is the round metric on $\Sn^{n-1}$.  We write $\widetilde\nabla$ for the gradient on the unit round sphere $\Sn^{n-1}$.  The volume element is
$$
        dv=w(t)\,dt\,d\sigma.
$$

\subsection{The first eigenfunctions of a hyperbolic ball}

We shall use the following standard spectral fact for geodesic balls in hyperbolic space.

\begin{lemma}[First Neumann eigenspace of a hyperbolic ball]\label{lem:hyperbolic-ball-facts}
Let $B_R\subset\HH^n$ be a geodesic ball of radius $R$ and let
$$
        \lambda=\mu_1(B_R).
$$
The first nonzero Neumann eigenvalue of $B_R$ has multiplicity $n$.  Its eigenspace is spanned by functions
$$
        u_i(t,\xi)=f(t)\xi_i,
        \qquad i=1,\dots,n,
$$
where $\xi=(\xi_1,\dots,\xi_n)\in\Sn^{n-1}\subset\RR^n$, and where $f$ is the regular solution of
\begin{equation}\label{eq:hyperbolic-ball-ode}
        f''+(n-1)\coth t\,f'
        +\left(\lambda-\frac{n-1}{s(t)^2}\right)f=0,
        \qquad 0<t<R,
\end{equation}
with boundary conditions
\begin{equation}\label{eq:hyperbolic-ball-bc}
        f(0)=0,
        \qquad
        f'(R)=0.
\end{equation}
After multiplication by a constant, $f$ may be chosen so that
$$
        f(t)>0\qquad 0<t\le R.
$$
Moreover,
\begin{equation}\label{eq:hyperbolic-ball-rayleigh}
        \lambda
        =\frac{\displaystyle\int_0^R\left(f'(t)^2+(n-1)\frac{f(t)^2}{s(t)^2}\right)w(t)\,dt}
        {\displaystyle\int_0^R f(t)^2w(t)\,dt},
\end{equation}
and
\begin{equation}\label{eq:hyperbolic-f-prime-positive}
        f'(t)>0\qquad 0<t<R.
\end{equation}
\end{lemma}

\begin{proof}
See Appendix \ref{app:first-neumann-balls}.
\end{proof}

Define the Weinberger extension
\begin{equation}\label{eq:hyperbolic-F-def}
        F(t)=
        \begin{cases}
        f(t),&0\le t\le R,\\
        f(R),&t\ge R.
        \end{cases}
\end{equation}
Then $F$ is nondecreasing, $F'(t)=0$ for $t>R$, and $F(0)=0$.  The function $F(t)/s(t)$ is understood at $t=0$ by taking its finite limit.  The following monotonicity is the key radial result.

\begin{lemma}\label{lem:hyperbolic-F-over-sinh}
The function
$$
        t\longmapsto \frac{F(t)}{\sinh t}
$$
is nonincreasing on $(0,\infty)$, and it is strictly decreasing on $(R,\infty)$.
\end{lemma}

\begin{proof}
For $t\ge R$, $F(t)=f(R)$, so $F(t)/\sinh t$ is strictly decreasing.  It remains to consider $0<t\le R$.
Set
$$
        \gamma(t)=f'(t)-\coth t\,f(t).
$$
The desired assertion on $(0,R]$ is equivalent to $\gamma(t)\le0$.  Since $f(t)=a_0t+O(t^3)$ near $0$ for some $a_0>0$, we have $\lim_{t\downarrow0}\gamma(t)=0$.  Also
$$
        \gamma(R)=f'(R)-\coth R\,f(R)=-\coth R\,f(R)<0.
$$
Suppose, for contradiction, that $\gamma(t_0)>0$ for some $t_0\in(0,R)$.  Then $\gamma$ attains a positive maximum at some $t_1\in(0,R)$, and hence
\begin{equation}\label{eq:hyperbolic-gamma-prime-zero}
        0=\gamma'(t_1)=f''(t_1)+\frac{f(t_1)}{s(t_1)^2}-\coth t_1\,f'(t_1).
\end{equation}
Using \eqref{eq:hyperbolic-ball-ode} at $t_1$ and eliminating $f''(t_1)$ by \eqref{eq:hyperbolic-gamma-prime-zero}, one obtains
$$
        f'(t_1)-\frac{f(t_1)}{\sinh t_1\cosh t_1}
        =-\frac{\lambda f(t_1)\sinh t_1}{n\cosh t_1}<0.
$$
Since
$$
        \coth t_1=\frac{\cosh t_1}{\sinh t_1}
        >\frac{1}{\sinh t_1\cosh t_1},
$$
this implies
$$
        f'(t_1)-\coth t_1\,f(t_1)<0,
$$
contradicting $\gamma(t_1)>0$.  Thus $\gamma\le0$, and the proof is complete.
\end{proof}

Introduce the radial constants
\begin{equation}\label{eq:hyperbolic-AQH-def}
        A=\int_0^R f(t)^2w(t)\,dt,
        \qquad
        Q=\int_0^R f'(t)^2w(t)\,dt,
        \qquad
        H=\int_0^R \frac{f(t)^2}{s(t)^2}w(t)\,dt.
\end{equation}
By \eqref{eq:hyperbolic-ball-rayleigh},
\begin{equation}\label{eq:hyperbolic-eigen-identity}
        Q+(n-1)H=\lambda A.
\end{equation}

\subsection{A hyperbolic Weinberger center}

\begin{lemma}\label{lem:hyperbolic-center}
There exists $p\in\HH^n$ such that, in geodesic polar coordinates centered at $p$,
\begin{equation}\label{eq:hyperbolic-center}
        \int_\Omega F(t)\xi\,dv=0\in T_p\HH^n.
\end{equation}
Equivalently, for every first spherical harmonic $\omega(\xi)=v\cdot\xi$,
$$
        \int_\Omega F(t)\omega(\xi)\,dv=0.
$$
\end{lemma}

\begin{proof}
Let
$$
        \Phi(\rho)=\int_0^\rho F(\tau)\,d\tau,
        \qquad
        J(q)=\int_\Omega \Phi(d(q,x))\,dv_x.
$$
The function $J$ is continuous.  Indeed, $F$ is bounded on bounded intervals, $\Phi$ is locally Lipschitz, and $\Omega$ is bounded.  Moreover, $J$ is proper on $\HH^n$: since $\Omega$ is bounded, $d(q,x)\to\infty$ uniformly in $x\in\Omega$ as $q\to\infty$; and since $F(t)=f(R)>0$ for $t\ge R$, the primitive $\Phi(\rho)$ grows linearly as $\rho\to\infty$.  Hence $J$ has a minimizer $p$ by Hopf--Rinow.

For $x\ne p$, write $x=\exp_p(t\xi)$.  The first variation formula for distance gives
$$
        \nabla_p d(p,x)=-\xi.
$$
At the exceptional point $x=p$, the function $q\mapsto \Phi(d(q,p))$ has directional derivative zero at $q=p$, because $F(0)=0$ and $F(t)=O(t)$ as $t\downarrow0$, whence $\Phi(t)=O(t^2)$.  Fix a compact neighbourhood $U$ of $p$ and a number $D$ such that $d(q,x)\le D$ for all $q\in U$ and $x\in\Omega$.  For $x\ne p$, the directional derivative in a unit direction is bounded in absolute value by $\|F\|_{L^\infty([0,D])}$.  Since $\Omega$ has finite volume, dominated convergence justifies differentiation under the integral sign.  For every $v\in T_p\HH^n$,
$$
        0=dJ_p(v)=-\left\langle \int_\Omega F(t)\xi\,dv,\,v\right\rangle.
$$
This proves \eqref{eq:hyperbolic-center}.
\end{proof}

\subsection{Ray estimates in the hyperbolic case}

Let
\begin{equation}\label{eq:hyperbolic-Y-R-def}
        Y_R=\int_0^R w(t)\,dt.
\end{equation}
For a measurable $E\subset(0,\infty)$ put
$$
        \mathcal Y(E)=\int_Ew(t)\,dt.
$$

\begin{corollary}\label{cor:hyperbolic-ray-scalar}
Let $E\subset(0,\infty)$ be measurable with $\mathcal Y(E)<\infty$, and set $y_E=\mathcal Y(E)$.  Then
\begin{align}
        \int_E F(t)^2w(t)\,dt
        &\ge A+f(R)^2\,(y_E-Y_R),
        \label{eq:hyperbolic-mass-ray}\\
        \int_E \frac{F(t)^2}{s(t)^2}w(t)\,dt
        &\le H+\frac{f(R)^2}{s(R)^2}\,(y_E-Y_R),
        \label{eq:hyperbolic-angular-ray}\\
        \int_E F'(t)^2w(t)\,dt
        &\le Q.
        \label{eq:hyperbolic-radial-ray}
\end{align}
\end{corollary}

\begin{proof}
Since $F$ is nondecreasing, $F^2$ is nondecreasing.  By Lemma \ref{lem:ray-rearrangement},
$$
        \int_EF^2w\,dt\ge \int_0^{\tau(y_E)}F^2w\,dt.
$$
The function
$$
        \Psi(y)=\int_0^{\tau(y)}F(t)^2w(t)\,dt
$$
is convex in $y$, because $\Psi'(y)=F(\tau(y))^2$ is nondecreasing.  Moreover $\Psi(Y_R)=A$ and $\Psi'(Y_R)=f(R)^2$.  Hence the tangent-line inequality for a convex function gives
$$
        \Psi(y_E)\ge A+f(R)^2(y_E-Y_R),
$$
which proves \eqref{eq:hyperbolic-mass-ray}.

By Lemma \ref{lem:hyperbolic-F-over-sinh}, the function $F^2/s^2$ is nonincreasing.  Lemma \ref{lem:ray-rearrangement} gives
$$
        \int_E\frac{F^2}{s^2}w\,dt
        \le \int_0^{\tau(y_E)}\frac{F^2}{s^2}w\,dt.
$$
Now
$$
        \Theta(y)=\int_0^{\tau(y)}\frac{F(t)^2}{s(t)^2}w(t)\,dt
$$
is concave in $y$, with $\Theta(Y_R)=H$ and $\Theta'(Y_R)=f(R)^2/s(R)^2$.  Therefore the tangent-line inequality for a concave function gives
$$
        \Theta(y_E)\le H+\frac{f(R)^2}{s(R)^2}(y_E-Y_R),
$$
which proves \eqref{eq:hyperbolic-angular-ray}.  Finally, $F'=f'$ on $(0,R)$ and $F'=0$ on $(R,\infty)$, so \eqref{eq:hyperbolic-radial-ray} follows from nonnegativity of $F'^2w$.
\end{proof}

\subsection{Proof in the hyperbolic case}

\begin{proposition}\label{prop:hyperbolic}
Let $\Omega\subset\HH^n$ be a bounded connected smooth domain, and let $B_R\subset\HH^n$ be a geodesic ball such that $|B_R|=|\Omega|$.  Then
\begin{equation}\label{eq:hyperbolic-conclusion}
        \sum_{j=1}^n\frac1{\mu_j(\Omega)}
        \ge
        \frac{n}{\mu_1(B_R)}.
\end{equation}
Equality holds if and only if $\Omega$ is a geodesic ball of radius $R$.
\end{proposition}

\begin{proof}
Put $\lambda=\mu_1(B_R)$ and let $f,F,A,Q,H$ be as above.  Choose the point $p\in\HH^n$ given by Lemma \ref{lem:hyperbolic-center}.  In polar coordinates centered at $p$, write $x=\exp_p(t\xi)$, $\xi\in\Sn^{n-1}$, and define
\begin{equation}\label{eq:hyperbolic-P-def}
        P_i(x)=F(t)\xi_i,
        \qquad i=1,\dots,n.
\end{equation}
At $t=0$ we set $P_i(p)=0$.  In normal coordinates $y=t\xi$ centered at $p$, the regularity of the degree-one radial solution gives $f(t)=a_0t+O(t^3)$, and hence $P_i(y)=a_0y_i+O(|y|^3)$ near $p$.  Away from $p$, the profile $F$ is piecewise smooth and is $C^1$ across $t=R$ because $f'(R)=0$.  Thus $P_i\in H^1(\Omega)$.  By the choice of $p$, each $P_i$ has zero mean over $\Omega$.

The functions $P_1,\dots,P_n$ are linearly independent on $\Omega$.  Indeed, suppose that $P_b:=\sum_i b_iP_i$ vanishes a.e. on $\Omega$ for some $b\in\RR^n$.  Since $P_b$ is continuous and $\Omega$ is open, $P_b$ vanishes everywhere on $\Omega$.  If $b\ne0$, then away from the single point $p$ its zero set is contained in the totally geodesic hypersurface determined by $b\cdot\xi=0$.  Such a hypersurface has empty interior in $\HH^n$ and cannot contain the nonempty open set $\Omega$.  Hence $b=0$.

Recall the definition of mass and stiffness matrices \eqref{eq:hyperbolic-MK-def}. By Lemma \ref{lem:trace-ritz},
\begin{equation}\label{eq:hyperbolic-ritz-main}
        \sum_{j=1}^n\frac1{\mu_j(\Omega)}
        \ge \tr(\mcK^{-1}\mcM).
\end{equation}

For $\xi\in\Sn^{n-1}$, let
$$
        E_\xi=\{t>0:\exp_p(t\xi)\in\Omega\}
$$
and
$$
        Y(\xi)=\int_{E_\xi}w(t)\,dt.
$$
By Fubini's theorem, $E_\xi$ is measurable and $Y(\xi)<\infty$ for a.e. $\xi$.  Since $|\Omega|=|B_R|$,
\begin{equation}\label{eq:hyperbolic-delta-zero}
        \int_{\Sn^{n-1}}(Y(\xi)-Y_R)\,d\sigma(\xi)=0.
\end{equation}
Set
\begin{equation}\label{eq:hyperbolic-Z-def}
        Z=\int_{\Sn^{n-1}}(Y(\xi)-Y_R)\,\xi\otimes\xi\,d\sigma(\xi).
\end{equation}
Then $Z$ is symmetric and, by \eqref{eq:hyperbolic-delta-zero},
\begin{equation}\label{eq:hyperbolic-Z-trace-zero}
        \tr Z=0.
\end{equation}

Let
\begin{equation}\label{eq:hyperbolic-aceta-def}
        a=\frac{|\Sn^{n-1}|}{n}A,
        \qquad
        c=f(R)^2,
        \qquad
        \eta=\frac{f(R)^2}{s(R)^2}.
\end{equation}
First,
$$
        \mcM
        =\int_{\Sn^{n-1}}
        \left(\int_{E_\xi}F(t)^2w(t)\,dt\right)
        \xi\otimes\xi\,d\sigma(\xi).
$$
By \eqref{eq:hyperbolic-mass-ray},
\begin{align*}
        \mcM
        &\succeq
        \int_{\Sn^{n-1}}
        \left(A+c(Y(\xi)-Y_R)\right)
        \xi\otimes\xi\,d\sigma(\xi)\\
        &=A\int_{\Sn^{n-1}}\xi\otimes\xi\,d\sigma+cZ
        =a\Id+cZ,
\end{align*}
where we used
\begin{equation}\label{eq:hyperbolic-sphere-id-1}
        \int_{\Sn^{n-1}}\xi\otimes\xi\,d\sigma=\frac{|\Sn^{n-1}|}{n}\Id.
\end{equation}

Second, since
$$
        \left\langle \widetilde\nabla \xi_i,\widetilde\nabla\xi_j\right\rangle_{\Sn^{n-1}}
        =\delta_{ij}-\xi_i\xi_j,
$$
the hyperbolic gradient decomposition gives
\begin{align*}
        \mcK
        &=\int_{\Sn^{n-1}}
        \left(\int_{E_\xi}F'(t)^2w(t)\,dt\right)
        \xi\otimes\xi\,d\sigma(\xi)\\
        &\quad+
        \int_{\Sn^{n-1}}
        \left(\int_{E_\xi}\frac{F(t)^2}{s(t)^2}w(t)\,dt\right)
        (\Id-\xi\otimes\xi)\,d\sigma(\xi).
\end{align*}
By \eqref{eq:hyperbolic-angular-ray} and \eqref{eq:hyperbolic-radial-ray},
\begin{align*}
        \mcK
        &\preceq
        \int_{\Sn^{n-1}}Q\,\xi\otimes\xi\,d\sigma\\
        &\quad+
        \int_{\Sn^{n-1}}
        \left(H+\eta(Y(\xi)-Y_R)\right)
        (\Id-\xi\otimes\xi)\,d\sigma.
\end{align*}
Using \eqref{eq:hyperbolic-delta-zero}, \eqref{eq:hyperbolic-Z-def}, and \eqref{eq:hyperbolic-sphere-id-1}, this becomes
\begin{align*}
        \mcK
        &\preceq
        \frac{|\Sn^{n-1}|}{n}Q\Id
        +\frac{(n-1)|\Sn^{n-1}|}{n}H\Id
        -\eta Z\\
        &=\frac{|\Sn^{n-1}|}{n}\bigl(Q+(n-1)H\bigr)\Id-\eta Z.
\end{align*}
By \eqref{eq:hyperbolic-eigen-identity},
\begin{equation}\label{eq:hyperbolic-K-bound-main}
        \mcK\preceq \lambda a\Id-\eta Z.
\end{equation}
Thus the hypotheses of Lemma \ref{lem:matrix} are satisfied.  Hence
\begin{equation}\label{eq:hyperbolic-trace-lower}
        \tr(\mcK^{-1}\mcM)\ge \frac{n}{\lambda}.
\end{equation}
Combining \eqref{eq:hyperbolic-ritz-main} and \eqref{eq:hyperbolic-trace-lower}, we obtain
$$
        \sum_{j=1}^n\frac1{\mu_j(\Omega)}
        \ge \frac{n}{\lambda}
        =\frac{n}{\mu_1(B_R)}.
$$
This proves the inequality.

If $\Omega$ is a geodesic ball of radius $R$, then
$$
        \mu_1(\Omega)=\cdots=\mu_n(\Omega)=\mu_1(B_R)=\lambda,
$$
so equality holds.

Conversely, suppose equality holds in \eqref{eq:hyperbolic-conclusion}.  Then equality holds in both inequalities
$$
        \sum_{j=1}^n\frac1{\mu_j(\Omega)}
        \ge \tr(\mcK^{-1}\mcM)
        \ge \frac{n}{\lambda}.
$$
In particular, equality holds in Lemma \ref{lem:matrix}.  By its equality statement,
\begin{equation}\label{eq:hyperbolic-equality-MKZ}
        Z=0,
        \qquad
        \mcM=a\Id,
        \qquad
        \mcK=\lambda a\Id.
\end{equation}
Since $Z=0$, the upper comparison \eqref{eq:hyperbolic-K-bound-main} reads $\mcK\preceq \lambda a\Id$; by \eqref{eq:hyperbolic-equality-MKZ} equality holds in this comparison.

For a.e. $\xi\in\Sn^{n-1}$, set
\begin{align*}
        q(\xi)&=\int_{E_\xi}F'(t)^2w(t)\,dt,\\
        h(\xi)&=\int_{E_\xi}\frac{F(t)^2}{s(t)^2}w(t)\,dt,\\
        \delta(\xi)&=Y(\xi)-Y_R.
\end{align*}
By Corollary \ref{cor:hyperbolic-ray-scalar},
$$
        \alpha(\xi):=Q-q(\xi)\ge0,
        \qquad
        \beta(\xi):=H+\eta\delta(\xi)-h(\xi)\ge0.
$$
A direct subtraction of the formula for $\mcK$ from its upper bound gives
\begin{equation}\label{eq:hyperbolic-K-gap-integral}
        \lambda a\Id-\mcK
        =\int_{\Sn^{n-1}}
        \left[\alpha(\xi)\,\xi\otimes\xi+\beta(\xi)(\Id-\xi\otimes\xi)\right]d\sigma(\xi).
\end{equation}
The left-hand side is zero.  The integrand in \eqref{eq:hyperbolic-K-gap-integral} is positive semidefinite.  Taking traces yields
$$
        0=\int_{\Sn^{n-1}}\left(\alpha(\xi)+(n-1)\beta(\xi)\right)d\sigma(\xi).
$$
Therefore
\begin{equation}\label{eq:hyperbolic-alpha-beta-zero}
        \alpha(\xi)=0,
        \qquad
        \beta(\xi)=0
        \quad\text{for a.e. }\xi.
\end{equation}
In particular, $q(\xi)=Q$ for a.e. $\xi$, that is,
\begin{equation}\label{eq:hyperbolic-q-equality}
        \int_{E_\xi}F'(t)^2w(t)\,dt
        =\int_0^R f'(t)^2w(t)\,dt
        \quad\text{for a.e. }\xi.
\end{equation}
Since $F'=f'$ on $(0,R)$, $F'=0$ on $(R,\infty)$, and $f'(t)>0$ for $0<t<R$, \eqref{eq:hyperbolic-q-equality} forces
\begin{equation}\label{eq:hyperbolic-slice-contains-ball}
        (0,R)\subset E_\xi
        \quad\text{up to }w(t)dt\text{-null sets, for a.e. }\xi.
\end{equation}
Since $w(t)>0$ for $t>0$, these are equivalently one-dimensional Lebesgue null sets.  Thus $Y(\xi)\ge Y_R$ for a.e. $\xi$.  But \eqref{eq:hyperbolic-delta-zero} says that the average of $Y(\xi)-Y_R$ is zero.  Therefore
$$
        Y(\xi)=Y_R
        \quad\text{for a.e. }\xi.
$$
Together with \eqref{eq:hyperbolic-slice-contains-ball}, this gives
$$
        E_\xi=(0,R)
        \quad\text{up to }w(t)dt\text{-null sets, for a.e. }\xi.
$$
Hence $\Omega$ agrees almost everywhere with the geodesic ball $B_R(p)$.

Finally, Lemma \ref{lem:ae-domain-equality} upgrades this almost-everywhere equality to equality as domains.  Thus $\Omega=B_R(p)$.
\end{proof}

\section{The spherical case}\label{sec4}

We realize $\Sn^n$ as the unit sphere in $\RR^{n+1}$.  For $E\subset\Sn^n$, write $-E=\{-x:x\in E\}$, where $x\mapsto -x$ is the antipodal map.

Throughout this section, $\Omega\subset\Sn^n$ is a connected smooth domain contained in an open hemisphere.  Let $B_R\subset\Sn^n$ be a geodesic ball such that $|B_R|=|\Omega|$.  Since $\Omega$ is contained in an open hemisphere, $0<R\le\pi/2$.  Put
$$
        \lambda=\mu_1(B_R),
        \qquad
        s(t)=\sin t,
        \qquad
        w(t)=s(t)^{n-1}.
$$
About any point $p\in\Sn^n$, geodesic polar coordinates identify $\Sn^n\setminus\{p,-p\}$ with $(0,\pi)\times\Sn^{n-1}$ and the metric is
\begin{equation}\label{eq:sphere-polar-metric}
        ds_{\Sn^n}^2=dt^2+s(t)^2d\sigma^2.
\end{equation}
The volume element is
$$
        dv=w(t)\,dt\,d\sigma.
$$

\subsection{The first eigenfunctions of a spherical ball}

Also, we shall use the following standard spectral fact for spherical caps of radius at most $\pi/2$.

\begin{lemma}[First Neumann eigenspace of a spherical cap]\label{lem:sphere-ball-facts}
For $0<R\le\pi/2$, the first nonzero Neumann eigenvalue of $B_R\subset\Sn^n$ has multiplicity $n$, and its eigenspace is spanned by functions
$$
        u_i(t,\xi)=f(t)\xi_i,
        \qquad i=1,\dots,n,
$$
where $\xi=(\xi_1,\dots,\xi_n)\in\Sn^{n-1}\subset\RR^n$, and where $f$ is the regular solution of
\begin{equation}\label{eq:sphere-ball-ode}
        f''+(n-1)\cot t\,f'
        +\left(\lambda-\frac{n-1}{s(t)^2}\right)f=0,
        \qquad 0<t<R,
\end{equation}
with boundary conditions
\begin{equation}\label{eq:sphere-ball-bc}
        f(0)=0,
        \qquad
        f'(R)=0.
\end{equation}
After multiplication by a constant, $f$ may be chosen so that
$$
        f(t)>0\qquad 0<t\le R.
$$
Moreover,
\begin{equation}\label{eq:sphere-ball-rayleigh}
        \lambda
        =\frac{\displaystyle\int_0^R\left(f'(t)^2+(n-1)\frac{f(t)^2}{s(t)^2}\right)w(t)\,dt}
        {\displaystyle\int_0^R f(t)^2w(t)\,dt},
\end{equation}
and
\begin{equation}\label{eq:sphere-f-prime-positive}
        f'(t)>0\qquad 0<t<R.
\end{equation}
\end{lemma}

\begin{proof}
See Appendix \ref{app:first-neumann-balls}.
\end{proof}

We shall also use the elementary lower bound
\begin{equation}\label{eq:sphere-lambda-ge-n}
        \lambda\ge n,
\end{equation}
with equality when $R=\pi/2$.  Indeed, let $h(t)=\sin t$.  Then $h$ solves the same radial equation with eigenvalue $n$.  Written in Sturm--Liouville form,
$$
        (wf')'+\left(\lambda-\frac{n-1}{s^2}\right)wf=0,
        \qquad
        (wh')'+\left(n-\frac{n-1}{s^2}\right)wh=0.
$$
Subtracting the two equations after multiplication by $h$ and $f$, respectively, gives
$$
        \bigl(w(hf'-fh')\bigr)'+(\lambda-n)wfh=0.
$$
Integrating from $0$ to $R$ and using $f'(R)=0$, $h'(R)=\cos R$, yields
$$
        (\lambda-n)\int_0^R w(t)f(t)\sin t\,dt
        =w(R)f(R)\cos R\ge0.
$$
The integral on the left is strictly positive, because $w(t)>0$, $f(t)>0$, and $\sin t>0$ for $0<t<R$.  This proves \eqref{eq:sphere-lambda-ge-n}.  If $R<\pi/2$, then the right-hand side is strictly positive, so $\lambda>n$.  If $R=\pi/2$, the right-hand side is zero, hence $\lambda=n$; in this endpoint case the regular radial solution is proportional to $\sin t$.

Introduce the radial constants
\begin{equation}\label{eq:sphere-AQH-def}
        A=\int_0^R f(t)^2w(t)\,dt,
        \qquad
        Q=\int_0^R f'(t)^2w(t)\,dt,
        \qquad
        H=\int_0^R \frac{f(t)^2}{s(t)^2}w(t)\,dt.
\end{equation}
By \eqref{eq:sphere-ball-rayleigh},
\begin{equation}\label{eq:sphere-eigen-identity}
        Q+(n-1)H=\lambda A.
\end{equation}

Define the half-sphere Weinberger extension
\begin{equation}\label{eq:sphere-F-plus-def}
        F_+(t)=
        \begin{cases}
        f(t),&0\le t\le R,\\
        f(R),&R\le t\le \pi/2.
        \end{cases}
\end{equation}

\begin{lemma}\label{lem:sphere-F-plus-over-sin}
The function $F_+$ is nondecreasing on $[0,\pi/2]$, and
$$
        t\longmapsto \frac{F_+(t)}{\sin t}
$$
is nonincreasing on $(0,\pi/2]$.
\end{lemma}

\begin{proof}
The monotonicity of $F_+$ follows from \eqref{eq:sphere-f-prime-positive} and its definition.  On $[R,\pi/2]$, the quotient $F_+(t)/\sin t=f(R)/\sin t$ is nonincreasing.  It remains to consider $0<t<R$.

Set
$$
        \gamma(t)=f'(t)-\cot t\,f(t).
$$
The desired assertion on $(0,R)$ is equivalent to $\gamma(t)\le0$.  Since $f(t)=a_0t+O(t^3)$ near $0$ for some $a_0>0$, we have $\lim_{t\downarrow0}\gamma(t)=0$.  If $R<\pi/2$, then $\gamma(R)=-\cot R\,f(R)<0$; if $R=\pi/2$, then $\gamma(R)=0$.

Suppose, for contradiction, that $\gamma$ has a positive value.  Then it attains a positive maximum at some $t_1\in(0,R)$.  At this point $\gamma'(t_1)=0$.  Combining this identity with \eqref{eq:sphere-ball-ode} gives
$$
        n\cot t_1\,\gamma(t_1)+(\lambda-n)f(t_1)=0.
$$
Since $0<t_1<R\le\pi/2$, $\cot t_1>0$, and by \eqref{eq:sphere-lambda-ge-n} we have $\lambda\ge n$.  The last displayed identity is impossible when $\gamma(t_1)>0$.  Hence $\gamma\le0$, and the proof is complete.
\end{proof}

Finally define the antipodally even radial profile
\begin{equation}\label{eq:sphere-F-even-def}
        F(t)=F_+\bigl(\min\{t,\pi-t\}\bigr),
        \qquad 0\le t\le\pi.
\end{equation}
Then
\begin{equation}\label{eq:sphere-F-even}
        F(\pi-t)=F(t),
        \qquad
        F(0)=F(\pi)=0.
\end{equation}

\subsection{A spherical Weinberger center}

\begin{lemma}\label{lem:sphere-center}
There exists $p\in\Sn^n$ such that, in geodesic polar coordinates centered at $p$,
\begin{equation}\label{eq:sphere-center}
        \int_\Omega F(t)\xi\,dv=0\in T_p\Sn^n.
\end{equation}
Equivalently, for every first spherical harmonic $\omega(\xi)=v\cdot\xi$ on the equator of $p$,
$$
        \int_\Omega F(t)\omega(\xi)\,dv=0.
$$
\end{lemma}

\begin{proof}
Let
$$
        \Phi(\rho)=\int_0^\rho F(\tau)\,d\tau,
        \qquad
        J(q)=\int_\Omega \Phi(d(q,x))\,dv_x,
        \qquad q\in\Sn^n.
$$
Since $\Sn^n$ is compact, $J$ has a minimizer $p$.

Near the two exceptional radii one has
$$
        F(t)=O(t)\quad (t\downarrow0),
        \qquad
        F(t)=O(\pi-t)\quad (t\uparrow\pi),
$$
because $f(t)=a_0t+O(t^3)$ at the pole and $F(t)=F_+(\pi-t)$ near the antipode.  Consequently $\Phi(t)=O(t^2)$ as $t\downarrow0$ and $\Phi(\pi)-\Phi(t)=O((\pi-t)^2)$ as $t\uparrow\pi$.  Hence $q\mapsto\Phi(d(q,x))$ is differentiable at $q=p$ for the exceptional points $x=p$ and $x=-p$, with derivative equal to zero there.

For $x\ne p,-p$, the map $q\mapsto d(q,x)$ is smooth in a neighbourhood of $q=p$.  Writing $x=\exp_p(t\xi)$, the first variation formula for distance gives
$$
        \nabla_p d(p,x)=-\xi.
$$
For all nonexceptional $x$, the directional derivative in direction $v\in T_p\Sn^n$ is bounded in absolute value by $\|F\|_\infty |v|$.  Since the exceptional points have zero volume and $\Omega$ has finite measure, dominated convergence justifies differentiation under the integral sign.  For every $v\in T_p\Sn^n$,
$$
        0=dJ_p(v)=-\left\langle \int_\Omega F(t)\xi\,dv,\,v\right\rangle.
$$
This proves \eqref{eq:sphere-center}.
\end{proof}

\subsection{Antipodal folding and ray estimates}

Let $p$ be the center given by Lemma \ref{lem:sphere-center}.  For $x=\exp_p(t\xi)$ with $t\in(0,\pi)$, define
\begin{equation}\label{eq:sphere-P-def}
        P_i(x)=F(t)\xi_i,
        \qquad i=1,\dots,n,
\end{equation}
and set $P_i(p)=P_i(-p)=0$.  Near $p$, normal coordinates give $P_i(y)=a_0y_i+O(|y|^3)$.  Near $-p$, write $r=\pi-t$ and use $F(t)=O(r)$; the functions are again locally Lipschitz in normal coordinates at $-p$.  Away from $p$ and $-p$ the profile is piecewise smooth and is $C^1$ across the radii $R$, $\pi/2$, and $\pi-R$ because $f'(R)=0$ and the even reflection has zero derivative at $\pi/2$ from both sides.  Hence $P_i\in H^1(\Omega)$.  By Lemma \ref{lem:sphere-center}, each $P_i$ has zero mean over $\Omega$.

If $x=\exp_p(t\xi)$ and $t>\pi/2$, then
$$
        -x=\exp_p((\pi-t)(-\xi)).
$$
Using \eqref{eq:sphere-F-even}, we obtain
\begin{equation}\label{eq:sphere-P-odd}
        P_i(-x)=-P_i(x).
\end{equation}

\begin{lemma}[Antipodal folding]\label{lem:sphere-folding}
Let
$$
        \Omega_+=\{x\in\Omega:d(p,x)<\pi/2\},
        \qquad
        \Omega_-=\{x\in\Omega:d(p,x)>\pi/2\},
$$
where the equatorial slice $\Omega\cap\{d(p,x)=\pi/2\}$ is discarded.  Define
\begin{equation}\label{eq:sphere-folded-domain}
        \widehat\Omega=\Omega_+\cup(-\Omega_-)
        \subset B_{\pi/2}(p).
\end{equation}
Then the folding map, equal to the identity on $\Omega_+$ and to the antipodal map on $\Omega_-$, is injective and measure preserving up to null sets.  In particular,
\begin{equation}\label{eq:sphere-folded-volume}
        |\widehat\Omega|=|\Omega|=|B_R|.
\end{equation}
Moreover, if $G$ is any integrable density satisfying $G(-x)=G(x)$, then
$$
        \int_\Omega G\,dv=\int_{\widehat\Omega}G\,dv.
$$
Consequently,
\begin{equation}\label{eq:sphere-folded-integrals}
        \int_\Omega P_iP_j\,dv
        =\int_{\widehat\Omega}P_iP_j\,dv,
        \qquad
        \int_\Omega\langle\nabla P_i,\nabla P_j\rangle\,dv
        =\int_{\widehat\Omega}\langle\nabla P_i,\nabla P_j\rangle\,dv.
\end{equation}
\end{lemma}

\begin{proof}
The equatorial slice is contained in a smooth hypersurface and has zero $n$-dimensional measure.  Since $\Omega$ is contained in an open hemisphere, it contains no pair of antipodal points.  Hence the folding map is injective after the equatorial null set is removed; otherwise two distinct points of $\Omega$ would be antipodal.  The identity and the antipodal map are isometries, so the folding map is measure preserving on the two pieces.  This proves \eqref{eq:sphere-folded-volume} and the integral identity for every antipodally invariant density $G$.  Finally, \eqref{eq:sphere-P-odd} implies that $P_iP_j$ is antipodally invariant.  Since the antipodal map is an isometry and $P_i\circ A=-P_i$ for the antipodal map $A$,  $\langle\nabla P_i,\nabla P_j\rangle$ is also antipodally invariant.  This proves \eqref{eq:sphere-folded-integrals}.
\end{proof}

For $\xi\in\Sn^{n-1}$, let
$$
        E_\xi=\{t\in(0,\pi/2):\exp_p(t\xi)\in\widehat\Omega\}
$$
and
$$
        Y(\xi)=\int_{E_\xi}w(t)\,dt,
        \qquad
        Y_R=\int_0^R w(t)\,dt.
$$
By Fubini's theorem, $E_\xi$ is measurable for a.e. $\xi$.  By \eqref{eq:sphere-folded-volume},
\begin{equation}\label{eq:sphere-delta-zero}
        \int_{\Sn^{n-1}}(Y(\xi)-Y_R)\,d\sigma(\xi)=0.
\end{equation}
Set
\begin{equation}\label{eq:sphere-Z-def}
        Z=\int_{\Sn^{n-1}}(Y(\xi)-Y_R)\,\xi\otimes\xi\,d\sigma(\xi).
\end{equation}
Then
\begin{equation}\label{eq:sphere-Z-trace-zero}
        \tr Z=0.
\end{equation}

\begin{corollary}\label{cor:sphere-ray-scalar}
Let $E\subset(0,\pi/2)$ be measurable and set $y_E=\int_Ew(t)\,dt$.  Then
\begin{align}
        \int_E F_+(t)^2w(t)\,dt
        &\ge A+f(R)^2\,(y_E-Y_R),
        \label{eq:sphere-mass-ray}\\
        \int_E \frac{F_+(t)^2}{s(t)^2}w(t)\,dt
        &\le H+\frac{f(R)^2}{s(R)^2}\,(y_E-Y_R),
        \label{eq:sphere-angular-ray}\\
        \int_E F_+'(t)^2w(t)\,dt
        &\le Q.
        \label{eq:sphere-radial-ray}
\end{align}
\end{corollary}

\begin{proof}
This is the same tangent-line argument as in Corollary \ref{cor:hyperbolic-ray-scalar}, now on the finite interval $(0,\pi/2)$.  Namely, $F_+^2$ is nondecreasing, while $F_+^2/s^2$ is nonincreasing by Lemma \ref{lem:sphere-F-plus-over-sin}.  Applying Lemma \ref{lem:ray-rearrangement} and using the convexity of
$$
        \Psi(y)=\int_0^{\tau(y)}F_+(t)^2w(t)\,dt
$$
and the concavity of
$$
        \Theta(y)=\int_0^{\tau(y)}\frac{F_+(t)^2}{s(t)^2}w(t)\,dt
$$
applied at $y=y_E$ gives \eqref{eq:sphere-mass-ray} and \eqref{eq:sphere-angular-ray}.  When $R<\pi/2$, the supporting lines are the ordinary tangent lines at $Y_R$.  When $R=\pi/2$, $Y_R=\mathcal V(\pi/2)$ is the right endpoint of the interval of admissible $y$; for the convex function $\Psi$ and the concave function $\Theta$, the left derivatives at this endpoint determine the required supporting lines, so the same inequalities remain valid.  Finally, $F_+'=f'$ on $(0,R)$ and $F_+'=0$ on $(R,\pi/2)$, so \eqref{eq:sphere-radial-ray} follows from nonnegativity of $F_+'^{\,2}w$.
\end{proof}

\subsection{Proof in the spherical case}

\begin{proposition}\label{prop:spherical}
Let $\Omega\subset\Sn^n$ be a connected smooth domain contained in an open hemisphere, and let $B_R\subset\Sn^n$ be a geodesic ball such that $|B_R|=|\Omega|$.  Then
\begin{equation}\label{eq:spherical-conclusion}
        \sum_{j=1}^n\frac1{\mu_j(\Omega)}
        \ge
        \frac{n}{\mu_1(B_R)}.
\end{equation}
Equality holds if and only if $\Omega$ is a geodesic ball of radius $R$.
\end{proposition}

\begin{proof}
Let $p$ be the spherical center from Lemma \ref{lem:sphere-center}.  Define $P_i$ by \eqref{eq:sphere-P-def}.  These functions are linearly independent on $\Omega$.  Indeed, if $P_b:=\sum_i b_iP_i$ vanishes a.e. on $\Omega$, then by continuity it vanishes on the open set $\Omega$.  For $b\ne0$, the zero set of $P_b$ outside the two points $p$ and $-p$ is contained in the totally geodesic hypersurface determined by $b\cdot\xi=0$, and hence has empty interior.  Thus $b=0$.  Define $\mcM_{ij}$ and $\mcK_{ij}$ as in \eqref{eq:hyperbolic-MK-def}. By Lemma \ref{lem:trace-ritz},
\begin{equation}\label{eq:sphere-ritz-main}
        \sum_{j=1}^n\frac1{\mu_j(\Omega)}
        \ge \tr(\mcK^{-1}\mcM).
\end{equation}
By Lemma \ref{lem:sphere-folding}, specifically \eqref{eq:sphere-folded-integrals}, the matrices $\mcM$ and $\mcK$ may be computed over the folded set $\widehat\Omega\subset B_{\pi/2}(p)$.

Let
\begin{equation}\label{eq:sphere-aceta-def}
        a=\frac{|\Sn^{n-1}|}{n}A,
        \qquad
        c=f(R)^2,
        \qquad
        \eta=\frac{f(R)^2}{s(R)^2}.
\end{equation}
On $\widehat\Omega$ the profile is $F_+(t)$, so
$$
        \mcM
        =\int_{\Sn^{n-1}}
        \left(\int_{E_\xi}F_+(t)^2w(t)\,dt\right)
        \xi\otimes\xi\,d\sigma(\xi).
$$
By \eqref{eq:sphere-mass-ray},
\begin{align*}
        \mcM
        &\succeq
        \int_{\Sn^{n-1}}
        \left(A+c(Y(\xi)-Y_R)\right)
        \xi\otimes\xi\,d\sigma(\xi)\\
        &=A\int_{\Sn^{n-1}}\xi\otimes\xi\,d\sigma+cZ
        =a\Id+cZ.
\end{align*}

For the stiffness matrix, using the gradient decomposition in \eqref{eq:sphere-polar-metric},
\begin{align*}
        \mcK
        &=\int_{\Sn^{n-1}}
        \left(\int_{E_\xi}F_+'(t)^2w(t)\,dt\right)
        \xi\otimes\xi\,d\sigma(\xi)\\
        &\quad+
        \int_{\Sn^{n-1}}
        \left(\int_{E_\xi}\frac{F_+(t)^2}{s(t)^2}w(t)\,dt\right)
        (\Id-\xi\otimes\xi)\,d\sigma(\xi).
\end{align*}
By \eqref{eq:sphere-angular-ray} and \eqref{eq:sphere-radial-ray},
\begin{align*}
        \mcK
        &\preceq
        \int_{\Sn^{n-1}}Q\,\xi\otimes\xi\,d\sigma\\
        &\quad+
        \int_{\Sn^{n-1}}
        \left(H+\eta(Y(\xi)-Y_R)\right)
        (\Id-\xi\otimes\xi)\,d\sigma.
\end{align*}
Using \eqref{eq:sphere-delta-zero}, \eqref{eq:sphere-Z-def}, and
$$
        \int_{\Sn^{n-1}}\xi\otimes\xi\,d\sigma=\frac{|\Sn^{n-1}|}{n}\Id,
$$
we get
\begin{align*}
        \mcK
        &\preceq
        \frac{|\Sn^{n-1}|}{n}Q\Id
        +\frac{(n-1)|\Sn^{n-1}|}{n}H\Id
        -\eta Z\\
        &=\frac{|\Sn^{n-1}|}{n}\bigl(Q+(n-1)H\bigr)\Id-\eta Z.
\end{align*}
By \eqref{eq:sphere-eigen-identity},
\begin{equation}\label{eq:sphere-K-bound-main}
        \mcK\preceq \lambda a\Id-\eta Z.
\end{equation}
Thus the hypotheses of Lemma \ref{lem:matrix} are satisfied, and hence
\begin{equation}\label{eq:sphere-trace-lower}
        \tr(\mcK^{-1}\mcM)\ge \frac{n}{\lambda}.
\end{equation}
Together with \eqref{eq:sphere-ritz-main}, this gives
$$
        \sum_{j=1}^n\frac1{\mu_j(\Omega)}
        \ge \frac{n}{\lambda}
        =\frac{n}{\mu_1(B_R)}.
$$

If $\Omega$ is a geodesic ball of radius $R\le\pi/2$, then $\mu_1(\Omega)=\cdots=\mu_n(\Omega)=\lambda$, so equality holds.

Conversely, suppose equality holds in \eqref{eq:spherical-conclusion}.  Then equality holds in Lemma \ref{lem:matrix}, and therefore
\begin{equation}\label{eq:sphere-equality-MKZ}
        Z=0,
        \qquad
        \mcM=a\Id,
        \qquad
        \mcK=\lambda a\Id.
\end{equation}
In particular, equality holds in the stiffness comparison \eqref{eq:sphere-K-bound-main}.

For a.e. $\xi\in\Sn^{n-1}$, set
\begin{align*}
        q(\xi)&=\int_{E_\xi}F_+'(t)^2w(t)\,dt,\\
        h(\xi)&=\int_{E_\xi}\frac{F_+(t)^2}{s(t)^2}w(t)\,dt,\\
        \delta(\xi)&=Y(\xi)-Y_R.
\end{align*}
By Corollary \ref{cor:sphere-ray-scalar},
$$
        \alpha(\xi):=Q-q(\xi)\ge0,
        \qquad
        \beta(\xi):=H+\eta\delta(\xi)-h(\xi)\ge0.
$$
Subtracting the formula for $\mcK$ from its upper bound gives
$$
        \lambda a\Id-\mcK
        =\int_{\Sn^{n-1}}
        \left[\alpha(\xi)\,\xi\otimes\xi+\beta(\xi)(\Id-\xi\otimes\xi)\right]d\sigma(\xi).
$$
The left-hand side is zero, and the integrand is positive semidefinite.  Taking traces gives
$$
        0=\int_{\Sn^{n-1}}\left(\alpha(\xi)+(n-1)\beta(\xi)\right)d\sigma(\xi).
$$
Since $\alpha,\beta\ge0$, it follows that
$$
        \alpha(\xi)=0,
        \qquad
        \beta(\xi)=0
        \quad\text{for a.e. }\xi.
$$
In particular,
$$
        \int_{E_\xi}F_+'(t)^2w(t)\,dt
        =\int_0^R f'(t)^2w(t)\,dt
        \quad\text{for a.e. }\xi.
$$
Since $F_+'=f'$ on $(0,R)$, $F_+'=0$ on $(R,\pi/2)$, and $f'(t)>0$ for $0<t<R$, it follows that
$$
        (0,R)\subset E_\xi
        \quad\text{up to }w(t)dt\text{-null sets, for a.e. }\xi.
$$
Since $w(t)>0$ for $t\in(0,\pi/2)$, these are equivalently one-dimensional Lebesgue null sets.  Thus $Y(\xi)\ge Y_R$ for a.e. $\xi$.  By \eqref{eq:sphere-delta-zero}, the average of $Y(\xi)-Y_R$ is zero, hence $Y(\xi)=Y_R$ for a.e. $\xi$.  Therefore
$$
        E_\xi=(0,R)
        \quad\text{up to }w(t)dt\text{-null sets, for a.e. }\xi,
$$
and consequently
\begin{equation}\label{eq:sphere-folded-ball-ae}
        \widehat\Omega=B_R(p)
        \quad\text{up to a null set}.
\end{equation}

Assume first that $R<\pi/2$.  Then $B_R(p)$ and $B_R(-p)$ are disjoint open balls.  We claim that
$$
        \Omega\subset B_R(p)\cup B_R(-p).
$$
Indeed, assume that some $x\in\Omega$ is outside this union.  Since $\Omega$ is open, and since the boundaries of $B_R(p)$ and $B_R(-p)$ are smooth hypersurfaces, a small neighbourhood of $x$ inside $\Omega$ contains a subset $U$ of positive measure outside $B_R(p)\cup B_R(-p)$.  The folding map is the identity on the northern side and the antipodal map on the southern side, hence it is measure preserving on each side of the equator.  After discarding the equatorial null set, the folded image of $U$ has positive measure and lies outside $B_R(p)$, contradicting \eqref{eq:sphere-folded-ball-ae}.  Thus the displayed inclusion holds.

Because the two balls are disjoint open sets and $\Omega$ is connected, $\Omega$ is contained in exactly one of them.  If $\Omega\subset B_R(p)$, then \eqref{eq:sphere-folded-ball-ae} gives $|\Omega\triangle B_R(p)|=0$.  If $\Omega\subset B_R(-p)$, then applying the antipodal map to \eqref{eq:sphere-folded-ball-ae} gives $|\Omega\triangle B_R(-p)|=0$.  Lemma \ref{lem:ae-domain-equality} now gives equality as domains.  Hence $\Omega$ is a geodesic ball.

Assume next that $R=\pi/2$.  By hypothesis there is an open hemisphere $\mathcal H$ such that $\Omega\subset\mathcal H$.  Since $|\Omega|$ equals the volume of a hemisphere and $\Omega\subset\mathcal H$, we have $|\mathcal H\setminus\Omega|=0$, hence $|\mathcal H\triangle\Omega|=0$.  Lemma \ref{lem:ae-domain-equality}, applied to $\mathcal H$ and $\Omega$, gives $\Omega=\mathcal H$.  Thus $\Omega$ is a hemisphere, which is a geodesic ball of radius $\pi/2$.
\end{proof}

\begin{proof}[Proof of Theorem \ref{thm:main}]
The hyperbolic assertion is Proposition \ref{prop:hyperbolic}, and the spherical assertion is Proposition \ref{prop:spherical}.
\end{proof}

\appendix

\section{First Neumann eigenspaces of geodesic balls}
\label{app:first-neumann-balls}

We prove Lemmas \ref{lem:hyperbolic-ball-facts} and
\ref{lem:sphere-ball-facts}.  The proof is by separation of variables and
elementary Sturm--Liouville comparison.

Let
$$
        s(t)=
        \begin{cases}
        \sinh t, & \text{in } \HH^n,\\
        \sin t,  & \text{in } \Sn^n,
        \end{cases}
        \qquad
        c(t)=\frac{s'(t)}{s(t)},
        \qquad
        w(t)=s(t)^{n-1}.
$$
Thus \(c'(t)=-s(t)^{-2}\).  In the spherical case below we assume
\(0<R\le \pi/2\).  In geodesic polar coordinates on a ball \(B_R\),
$$
        ds^2=dt^2+s(t)^2d\sigma^2,
        \qquad
        dv=w(t)\,dt\,d\sigma,
$$
and
$$
        \Delta
        =
        \partial_t^2+(n-1)c(t)\partial_t
        +\frac1{s(t)^2}\Delta_{\Sn^{n-1}}.
$$

Let \(Y_{\ell,m}\), \(m=1,\dots,d_\ell\), be spherical harmonics of degree
\(\ell\) on \(\Sn^{n-1}\), normalized so that
$$
        -\Delta_{\Sn^{n-1}}Y_{\ell,m}
        =
        \alpha_\ell Y_{\ell,m},
        \qquad
        \alpha_\ell=\ell(\ell+n-2).
$$
In particular,
$$
        \alpha_0=0,
        \qquad
        \alpha_1=n-1,
        \qquad
        d_1=n.
$$

Separation of variables reduces the Neumann problem on \(B_R\) to the
one-dimensional problems
\begin{equation}
\label{eq:appendix-sector-equation}
        -(w\varphi')'
        +
        \frac{\alpha_\ell}{s^2}w\varphi
        =
        \nu w\varphi,
        \qquad
        \varphi'(R)=0,
\end{equation}
with the regularity condition at \(t=0\).  Equivalently,
$$
        \varphi''
        +(n-1)c\varphi'
        +
        \left(\nu-\frac{\alpha_\ell}{s^2}\right)\varphi
        =
        0.
$$
For fixed \(\ell\), write the corresponding eigenvalues as
$$
        \nu_{\ell,1}\le\nu_{\ell,2}\le\cdots .
$$
For \(\ell=0\), \(\nu_{0,1}=0\), with constant eigenfunction.  By the
standard one-dimensional Sturm--Liouville theory for this regular-singular
endpoint problem, the first eigenvalue in each angular sector is simple in
the radial variable, and its radial eigenfunction has no zero in \((0,R)\).
Moreover a regular eigenfunction in the \(\ell\)-th sector has expansion
$$
        \varphi(t)=a t^\ell+O(t^{\ell+2})
        \qquad (t\downarrow0).
$$

\begin{lemma}
\label{lem:appendix-sector-ordering-short}
For a hyperbolic ball of any radius \(R>0\), and for a spherical cap with
\(0<R\le\pi/2\),
$$
        \nu_{1,1}<\nu_{0,2},
        \qquad
        \nu_{1,1}<\nu_{\ell,1}
        \quad\text{for every }\ell\ge2.
$$
Hence the first nonzero Neumann eigenvalue of \(B_R\) is \(\nu_{1,1}\).
\end{lemma}

\begin{proof}
For the comparison with higher angular sectors, let \(\ell\ge2\).  Since
\(\alpha_\ell>\alpha_1=n-1\), for every admissible nonzero \(\varphi\),
$$
        \mathcal R_\ell[\varphi]
        =
        \mathcal R_1[\varphi]
        +
        (\alpha_\ell-\alpha_1)
        \frac{
        \displaystyle\int_0^R \frac{\varphi^2}{s^2}w\,dt
        }{
        \displaystyle\int_0^R \varphi^2w\,dt
        },
$$
where
$$
        \mathcal R_\ell[\varphi]
        =
        \frac{
        \displaystyle\int_0^R
        \left(\varphi'^2+\frac{\alpha_\ell}{s^2}\varphi^2\right)w\,dt
        }{
        \displaystyle\int_0^R \varphi^2w\,dt
        }.
$$
Taking \(\varphi\) to be the positive radial ground state in the
\(\ell\)-th sector gives
$$
        \nu_{\ell,1}
        =
        \mathcal R_\ell[\varphi]
        >
        \mathcal R_1[\varphi]
        \ge
        \nu_{1,1}.
$$

It remains to compare the degree-one sector with the radial sector.  Let
\(g\) be an eigenfunction for \(\nu_{0,2}\).  Then
$$
        g''+(n-1)cg'+\nu_{0,2}g=0,
        \qquad
        g'(0)=g'(R)=0.
$$
Set \(h=g'\).  Since \(g\) is nonconstant, \(h\not\equiv0\).  Differentiating
the equation for \(g\) and using \(c'=-s^{-2}\), we obtain
$$
        h''
        +(n-1)ch'
        +
        \left(\nu_{0,2}-\frac{n-1}{s^2}\right)h
        =
        0.
$$
Equivalently,
$$
        -(wh')'
        +
        \frac{n-1}{s^2}wh
        =
        \nu_{0,2}wh.
$$
Moreover \(h(R)=g'(R)=0\), and the regularity at the pole gives
\(h(t)=O(t)\).  Multiplying the last equation by \(h\), integrating by parts,
and using the vanishing of the boundary terms at \(0\) and \(R\), we get
$$
        \int_0^R
        \left(h'^2+\frac{n-1}{s^2}h^2\right)w\,dt
        =
        \nu_{0,2}
        \int_0^R h^2w\,dt.
$$
Therefore
$$
        \nu_{1,1}\le \mathcal R_1[h]=\nu_{0,2}.
$$

If equality held, then \(h\) would minimize the degree-one Rayleigh quotient.
Hence \(h\) would satisfy the natural Neumann boundary condition
$$
        h'(R)=0.
$$
Together with \(h(R)=0\), this gives zero Cauchy data for a nonsingular ODE
at \(t=R\).  Thus \(h\equiv0\), a contradiction.  Hence
$$
        \nu_{1,1}<\nu_{0,2}.
$$
The first nonzero Neumann eigenvalue is therefore \(\nu_{1,1}\).
\end{proof}

\begin{lemma}
\label{lem:appendix-f-prime-positive-short}
Let \(f\) be the positive radial ground state in the degree-one sector:
$$
        f''
        +(n-1)cf'
        +
        \left(\nu_{1,1}-\frac{n-1}{s^2}\right)f
        =
        0,
        \qquad
        f(0)=0,
        \qquad
        f'(R)=0.
$$
Then, in the hyperbolic case for every \(R>0\), and in the spherical case
for \(0<R\le\pi/2\),
$$
        f'(t)>0
        \qquad
        0<t<R.
$$
\end{lemma}

\begin{proof}
Put
$$
        \lambda=\nu_{1,1},
        \qquad
        m(t)=w(t)f'(t).
$$
The equation for \(f\) gives
$$
        m'(t)
        =
        \left(\frac{n-1}{s(t)^2}-\lambda\right)w(t)f(t).
$$
Also
$$
        m(0)=0,
        \qquad
        m(R)=0.
$$
Since \(f>0\) and \(w>0\) on \((0,R)\), the sign of \(m'\) is the sign of
$$
        \rho(t)=\frac{n-1}{s(t)^2}-\lambda.
$$
In the hyperbolic case \(s(t)=\sinh t\), and in the spherical case with
\(R\le\pi/2\), \(s(t)=\sin t\) is strictly increasing on \((0,R)\).  Hence
\(\rho\) is strictly decreasing.

We claim that \(\rho(R)<0\).  If \(\rho(R)\ge0\), then \(\rho(t)>0\) for
all \(0<t<R\).  Hence \(m'(t)>0\) on \((0,R)\), which contradicts
\(m(0)=m(R)=0\).  Thus \(\rho(R)<0\).  Since \(\rho(t)\to+\infty\) as
\(t\downarrow0\), there is a unique \(\tau\in(0,R)\) such that
\(\rho(\tau)=0\).  Therefore
$$
        m'>0 \text{ on }(0,\tau),
        \qquad
        m'<0 \text{ on }(\tau,R).
$$
It follows that \(m(t)>0\) on \((0,\tau]\).  For \(\tau<t<R\),
$$
        m(t)=-\int_t^R m'(r)\,dr>0.
$$
Thus \(m(t)>0\) for every \(0<t<R\).  Since \(w(t)>0\), we obtain
$$
        f'(t)=\frac{m(t)}{w(t)}>0
        \qquad
        0<t<R.
$$
\end{proof}

\begin{proof}[Proof of Lemmas \ref{lem:hyperbolic-ball-facts}
and \ref{lem:sphere-ball-facts}]
By Lemma \ref{lem:appendix-sector-ordering-short}, the first nonzero
Neumann eigenvalue of \(B_R\) is the first eigenvalue in the degree-one
sector:
$$
        \mu_1(B_R)=\nu_{1,1}.
$$
The radial ground state in this sector is simple, while the degree-one
spherical harmonics on \(\Sn^{n-1}\) have dimension \(n\) and are spanned by
the coordinate functions
$$
        \xi_1,\dots,\xi_n.
$$
Therefore the whole first nonzero Neumann eigenspace is
$$
        \spanop\{f(t)\xi_1,\dots,f(t)\xi_n\}.
$$
Thus the multiplicity is \(n\), and the eigenfunctions have the form
$$
        u_i(t,\xi)=f(t)\xi_i,
        \qquad i=1,\dots,n.
$$

In the hyperbolic case \(s(t)=\sinh t\) and \(c(t)=\coth t\), so the radial
equation is
$$
        f''+(n-1)\coth t\,f'
        +
        \left(\lambda-\frac{n-1}{\sinh^2t}\right)f=0.
$$
In the spherical case \(s(t)=\sin t\) and \(c(t)=\cot t\), so the radial
equation is
$$
        f''+(n-1)\cot t\,f'
        +
        \left(\lambda-\frac{n-1}{\sin^2t}\right)f=0.
$$
In both cases, regularity at the pole gives \(f(0)=0\), the Neumann boundary
condition gives \(f'(R)=0\), and the ground-state property gives
$$
        f(t)>0
        \qquad
        0<t\le R.
$$
By Lemma \ref{lem:appendix-f-prime-positive-short},
$$
        f'(t)>0
        \qquad
        0<t<R.
$$

Finally, multiplying the radial equation
$$
        (wf')'
        +
        \left(\lambda-\frac{n-1}{s^2}\right)wf
        =
        0
$$
by \(f\), integrating over \((0,R)\), and using the regularity at \(0\) and
the boundary condition \(f'(R)=0\), we obtain
$$
        \int_0^R
        \left(
        f'(t)^2+(n-1)\frac{f(t)^2}{s(t)^2}
        \right)w(t)\,dt
        =
        \lambda
        \int_0^R f(t)^2w(t)\,dt.
$$
Hence
$$
        \lambda
        =
        \frac{\displaystyle\int_0^R
        \left(
        f'(t)^2+(n-1)\frac{f(t)^2}{s(t)^2}
        \right)w(t)\,dt}
        {\displaystyle\int_0^R f(t)^2w(t)\,dt}.
$$
This proves both Lemma \ref{lem:hyperbolic-ball-facts} and Lemma
\ref{lem:sphere-ball-facts}.
\end{proof}


\begin{thebibliography}{99}

\bibitem{AshbaughOpen}
M. S. Ashbaugh,
\emph{Open problems on eigenvalues of the Laplacian},
in T. M. Rassias and H. M. Srivastava (eds.),
\emph{Analytic and Geometric Inequalities and Applications},
Math. Appl. \textbf{478}, Kluwer Academic Publishers, Dordrecht, 1999, pp. 13--28.

\bibitem{AB1993}
M. S. Ashbaugh and R. D. Benguria,
\emph{Universal bounds for the low eigenvalues of Neumann Laplacians in $N$ dimensions},
SIAM J. Math. Anal. \textbf{24} (1993), no. 3, 557--570.

\bibitem{AB1995}
M. S. Ashbaugh and R. D. Benguria,
\emph{Sharp upper bound to the first nonzero Neumann eigenvalue for bounded domains in spaces of constant curvature},
J. London Math. Soc. (2) \textbf{52} (1995), no. 2, 402--416.

\bibitem{BBC}
R. D. Benguria, B. Brandolini and F. Chiacchio,
\emph{A sharp estimate for Neumann eigenvalues of the Laplace--Beltrami operator for domains in a hemisphere},
Commun. Contemp. Math. \textbf{22} (2020), no. 3, 1950018, 9 pp.

\bibitem{BrascoPratelli}
L. Brasco and A. Pratelli,
\emph{Sharp stability of some spectral inequalities},
Geom. Funct. Anal. \textbf{22} (2012), no. 1, 107--135.

\bibitem{Chavel}
I. Chavel,
\emph{Eigenvalues in Riemannian Geometry},
Pure and Applied Mathematics, vol. 115, Academic Press, Orlando, FL, 1984.

\bibitem{HeLiTang}
Y. He, Y. Li and Q. Tang,
\emph{A proof of the Ashbaugh--Benguria conjecture for reciprocal sums of Neumann eigenvalues},
arXiv:2606.08271 [math.SP], 2026, 15 pp.

\bibitem{HileXu}
G. N. Hile and Z. Y. Xu,
\emph{Inequalities for sums of reciprocals of eigenvalues},
J. Math. Anal. Appl. \textbf{180} (1993), no. 2, 412--430.

\bibitem{Henrot}
A. Henrot,
\emph{Extremum Problems for Eigenvalues of Elliptic Operators},
Frontiers in Mathematics, Birkh\"auser, Basel, 2006.

\bibitem{MengWang}
Y. Meng and K. Wang,
\emph{Isoperimetric inequalities for Neumann eigenvalues on bounded domains in rank-1 symmetric spaces},
Calc. Var. Partial Differential Equations \textbf{63} (2024), no. 5, Paper No. 113, 12 pp.

\bibitem{Nadirashvili}
N. Nadirashvili,
\emph{Conformal maps and isoperimetric inequalities for eigenvalues of the Neumann problem},
in: \emph{Proceedings of the Ashkelon Workshop on Complex Function Theory (1996)},
Israel Math. Conf. Proc. \textbf{11}, Bar-Ilan Univ., Ramat Gan, 1997, pp. 197--201.

\bibitem{Szego}
G. Szeg\H{o},
\emph{Inequalities for certain eigenvalues of a membrane of given area},
J. Rational Mech. Anal. \textbf{3} (1954), 343--356.

\bibitem{Weinberger}
H. F. Weinberger,
\emph{An isoperimetric inequality for the $N$-dimensional free membrane problem},
J. Rational Mech. Anal. \textbf{5} (1956), no. 4, 633--636.

\bibitem{XiaWang}
C. Xia and Q. Wang,
\emph{On a conjecture of Ashbaugh and Benguria about lower eigenvalues of the Neumann Laplacian},
Math. Ann. \textbf{385} (2023), no. 1--2, 863--879.

\bibitem{Xu}
Y. Xu,
\emph{The first nonzero eigenvalue of Neumann problem on Riemannian manifolds},
J. Geom. Anal. \textbf{5} (1995), no. 1, 151--165.

\end{thebibliography}
\end{document}